\documentclass[letterpaper, 10 pt, conference]{ieeeconf}  % Comment this line out if you need a4paper

\IEEEoverridecommandlockouts                              % This command is only needed if 
\usepackage[dvipsnames]{xcolor}
\usepackage{graphicx}
\usepackage{soul}
\usepackage{dsfont}
\usepackage{mathtools}
\usepackage{wrapfig}
\usepackage{amssymb, amsmath, latexsym}
\usepackage{nicefrac}
\usepackage{hhline}
\usepackage{multirow}
\usepackage{hyperref}
\usepackage{pifont}% http://ctan.org/pkg/pifont
\usepackage[colorinlistoftodos,bordercolor=orange,backgroundcolor=orange!20,linecolor=orange,textsize=scriptsize]{todonotes}
\usepackage{algorithm}\usepackage{algpseudocode}
\usepackage[symbol]{footmisc}

\usepackage{tikz}
\usetikzlibrary{matrix, positioning, fit}

\newcommand{\argmin}{\mathop{\arg\!\min}}

\def \RR {\mathbb R}

\newcommand{\EE}{\mathbf{E}}

\def\EE{\mathbb E}

\def\e{\varepsilon}

\newtheorem{assumption}{Assumption}
\newtheorem{theorem}{Theorem}
 
\title{\LARGE \bf
An Accelerated Second-Order Method for 
Distributed Stochastic Optimization
}

\author{Artem Agafonov, Pavel Dvurechensky, Gesualdo Scutari, Alexander Gasnikov, \\ Dmitry Kamzolov, Aleksandr Lukashevich, and Amir Daneshmand% <-this % stops a space
\thanks{
A.A. and D.K. are with the Moscow Institute of Physics and Technology, 
Russia
{\tt\small ({agafonov.ad@phystech.edu}, {dkamzolov@yandex.ru})}. 
P.D. is with the Weierstrass Institute for Applied Analysis and Stochastics, 
Germany {(\tt\small pavel.dvurechensky@wias- berlin.de)}.
G.S. and A.D. are with the School of Industrial Engineering, Purdue University, 
USA {\tt\small ({gscutari@purdue.edu}, {adaneshm@purdue.edu})}.
A.G. is with  Moscow Institute of Physics and Technology, Russia, 
Institute for Information Transmission Problems RAS, 
Russia,  
Higher school of economics, 
Russia {(\tt\small gasnikov@yandex.ru)}.
A.L. is with the Skolkovo Institute of Science and Technology,
Russia,
Russia {(\tt\small aleksandr.lukashevich@skoltech.ru
)}.
}%
}

\begin{document}

\maketitle
\thispagestyle{empty}
\pagestyle{empty}

%%%%%%%%%%%%%%%%%%%%%%%%%%%%%%%%%%%%%%%%%%%%%%%%%%%%%%%%%%%%%%%%%%%%%%%%%%%%%%%%
%
\begin{abstract}
We consider distributed stochastic optimization problems that are solved with master/workers computation architecture. Statistical arguments allow to exploit statistical similarity and approximate this problem by a finite-sum problem, for which we propose an inexact accelerated cubic-regularized Newton's method that achieves lower communication complexity bound for this setting and improves upon existing upper bound. We further exploit this algorithm to obtain convergence rate bounds for the original stochastic optimization problem and compare our bounds with the existing bounds in several regimes when the goal is to minimize the number of communication rounds and increase the parallelization by increasing the number of workers.  
%inding the global minimum of convex (moreover, nondifferentiable) functions is one of the main and most difficult problems in modern optimization. Such problems arise in many applied sciences, from physics to deep learning. In this paper, we consider a certain class of "good"{} non-convex functions that can be bounded above and below by a parabolic function. We show that using only the zeroth-order oracle, one can obtain the linear speed $\log \left(\nicefrac{1}{\varepsilon}\right)$ of finding the global minimum on a cube. 
% In this article, we \ag{investigate
% }
% % consider
% an accelerated first-order method, namely, the method of similar triangles, which is optimal in the class of convex (strongly convex) problems with a \ag{Lipschitz} gradient. The paper considers a model of additive noise in a gradient and a Euclidean prox-structure \ag{for not necessarily bounded sets}. Convergence estimates are obtained in the case of strong convexity and its absence, and a stopping criterion is proposed for not strongly convex problems.

\end{abstract}

\section{Introduction}
Distributed optimization lies on the interface between control and optimization with the goal being to find a minimum of some global objective by a network of agents, each of which has access to a local part of the objective and can interact with only neighbouring agents. 
Many algorithms for this setting under different assumptions were proposed as early as in 1970s \cite{bor82,tsi84,deg74}. Moreover this setting has many applications including robotics, resource allocation, power system control, control of drone or satellite networks, distributed statistical inference and multiagent reinforcement learning \cite{xia06,rab04,ram2009distributed,kra13,nedic2017fast}.  
An important recent application is training large-scale machine learning models, which in the language of optimization requires to solve distributed stochastic optimization problems. 
%\cite{bot10,boy11,aba16,ned16w,ned15}

%\todo{Maybe we need a paragraph on why distributed optimization is important, how it is connected to learning community, cite some classical papers.} 

This paper focuses on distributed stochastic optimization problems using master/workers architectures. These computational architectures are common, e.g., in the context of federated learning \cite{woodworth2018graph,kairouz2019advances}, where for privacy-preserving purposes the dataset is split across multiple workers and computations are coordinated by the master node. %This approach allows to scale-up computations and learn modern large-scale Machine Learning (ML) models using distributed datasets. %In particular, recent advances in this direction are connected with the approach called Federated Learning \cite{woodworth2018graph,kairouz2019advances}, where each worker runs some optimization method and then the master node averages the iterates obtained by the workers.
To be more precise, we consider  the following general stochastic optimization problem:
\begin{equation}\label{StocProb}
    \min_{x\in\RR^d}\boldsymbol{F}(x):=\EE_{\xi}{f(x,\xi)},
\end{equation}
where $\xi$ is a random variable, e.g. random data, $f$ is convex and sufficiently smooth, which implies that $\boldsymbol{F}$ is convex. 
We assume that we have access to $m$ workers, $T$ rounds of communications (all to all or to the master node), and a total fixed budget of $N$ realizations of $\xi$. Under this assumption the main question is how small we can make the error $\EE{\boldsymbol{F}(x^T)} - \boldsymbol{F}(x^*)$ by different algorithms returning a random point $x^T$.
Here $x^*$ denotes a solution to \eqref{StocProb}.

%and $n$ -- number of local steps at each node in between communication rounds, meaning that between two consecutive communication rounds each worker has access to $n$ stochastic gradients  $\nabla f(x,\xi)$. Under this assumption the total budget of realizations of $\xi$ is fixed to be equal to \aa{$N$} and the main question is how small we can make the error $\EE{\boldsymbol{F}(x^T)} - \boldsymbol{F}(x^*)$ by different algorithms returning a random point $x^T$.

To solve \eqref{StocProb} on master/workers architectures, two main approaches are used \cite{nemirovski2009robust,shapiro2014lectures,dvinskikh2020stochastic}, namely Stochastic Approximation (SA) and Sample Average Approximation (SAA), a.k.a.  Monte-Carlo. 
The division between SA and SAA approaches is made for simplicity and there are algorithms (see \cite{li2020optimal} and Appendix \ref{S:SOTA}) which are based on the SAA idea, but, in fact, use a small number of stochastic gradient for each realization of $\xi$, which makes them quite close to SA algorithms. Moreover, in most cases, we are given a dataset and it is our choice whether we see each example once as in the SA approach or multiple times as in the SAA approach. %\textcolor{orange}{are you sure about that? SA is based on iid assumption over the entire iterates, this is not the case for SAA.} \pd{Pavel: I think yes: if we make one pass over the dataset, then we will have iid observations as in SA approach.} \textcolor{orange}{yes, I agree, for the specific setting where you do not revisit the data}
In this paper, we make an attempt to look at SA and SAA approaches from a unified perspective of the actual goal being to solve the stochastic optimization problem \eqref{StocProb} with some fixed budget $N$ of realizations of random variable $\xi$.

\subsubsection{Stochastic Approximation}

% Modern data since problems have to deal with a lot of data. So distributed algorithms are used \cite{nedic2020distributed}. It is well known that many of data since problems reduce to stochastic optimization problem of the type \cite{shalev2014understanding}
% \begin{equation}\label{StocProb}
%     \min_{x\in\RR^d}f(x):=\EE{f(x,\xi)}.
% \end{equation}
% In this paper we additionally assume that \eqref{StocProb} is convex and sufficiently smooth.

%For SA approach we consider a Federated Learning type centralized architecture \cite{woodworth2018graph,kairouz2019advances}. 
% Consider Federated Learning type of communication network architecture for smooth stochastic convex optimization problem $$\min_{x\in\RR^d}F(x):=\EE{f(x,\xi)}.$$
%In this setting the main assumptions are that we have access to $m$ workers, $T$ rounds of communications (all to all or to the master node) and $n$ -- number of local steps at each node in between communication rounds, meaning that between two consecutive communication rounds each worker has access to $n$ stochastic gradients  $\nabla f(x,\xi)$. 
% First we briefly overview the state of the art and then describe our suggestions on improving upon the existing results. 
In the SA approach a typical situation is that the total budget of $N$ realizations of $\xi$ is distributed between $T$ communication rounds and $m$ workers.
%each worker sequentially samples iid realizations of $\xi$ and corresponding stochastic gradient, and makes stochastic gradient step. 
This leads to so-called intermittent communications with $n=N/(mT)$ local stochastic gradient steps by each worker in between communication rounds, meaning that between two consecutive communication rounds each worker has access to $n$ iid stochastic gradients $\nabla f(x,\xi)$.
The authors of \cite{woodworth2021minmax} recently obtained for the setting of smooth optimization the following lower bound:\footnote{If $f$ is a quadratic the last term can be eliminated \cite{woodworth2020minibatch} and such bound is tight.}\footnote{Here and below we use $ \gtrsim$ and $\simeq$ 
  %in order not to write explicitly
  %$x^T$ is the best point generated by a method and 
  %we omitted 
  for simplicity to highlight the dependence on our main parameters $m,N,T$ and omit numerical multiplicative constants, logarithmic factors, and other parameters characterizing the problem, e.g. Lipschitz constants of the objective, its gradient, and Hessian, as well as estimates for the norm of the solution.} 
% \ag{I guess we should add here (and indicate that for simplicity we use this style below) some comments about the ignoring of all constants and log-factors...}
 \begin{equation}\label{Wood2021}
 \EE{\boldsymbol{F}(x^T)} - \boldsymbol{F}(x^*) \gtrsim  \frac{1}{\left(N/m\right)^2} + \frac{1}{\sqrt{N}} + \min\left\{\frac{1}{T^2},\frac{1}{\sqrt{N/m}}\right\}.
  \end{equation}
%   \ag{Let's use $\lesssim$ and $\gtrsim$ instead of $\le$ and $\ge$ in such inequalities...}
% In \cite{woodworth2021minmax} the following lower bound was obtained (we omit constant logarithmic parameters and different constant like Lipschitz (gradient) constant)\footnote{If $F$ is a quadratic the last term can be eliminated \cite{woodworth2020minibatch} and such bound is tight.}
%     $$\EE{F(x^N)} - F(x^*) \ge \frac{1}{R^2K^2} + \frac{\sigma}{\sqrt{MKR}} + \min\left\{\frac{1}{R^2},\frac{\sigma}{\sqrt{KR}}\right\}.$$
They show also that this bound is tight by showing that a special combination of Minibatched Accelerated SGD and Single-Machine Accelerated SGD achieve the bound \eqref{Wood2021}. 
%And from the other side \eqref{Wood2021} was obtained as a lower bound when assuming that $f$ is smooth (has Lipschitz gradient) and we have an access to $\nabla f(x,\xi)$ only at one point $x$, when $\xi$ is chosen. 
Under more restrictive assumptions on the smoothness of $f$ they obtain a lower bound which gives some room for improvement in the last term of \eqref{Wood2021}:

\begin{equation}\label{eq:lb_Lip_Hess}
    \begin{aligned}
        \min\left\{\frac{1}{T^2},\frac{1}{\sqrt{N/m}},\frac{1}{\left(N/m\right)^{1/4}T^{7/4}}\right\}\\
        \text{if Hessian is Lipschitz}, 
    \end{aligned}
\end{equation}

\begin{equation}\label{eq:lb_qscf}
    \begin{aligned}
     \min\left\{\frac{1}{T^2},\frac{1}{\sqrt{N/m}},\frac{1}{TN/m}\right\}\\
     \text{if $f$ is self-concordant}.   
    \end{aligned}
\end{equation}

% \begin{align}
% % &\min\left\{\frac{1}{T^2},\frac{1}{\sqrt{N/m}},\frac{1}{\left(N/m\right)^{1/4}T^{7/4}}\right\}~~\text{if Hessian is Lipschitz}, \label{eq:lb_Lip_Hess} \\
% \min\left\{\frac{1}{T^2},\frac{1}{\sqrt{N/m}},\frac{1}{TN/m}\right\}~~ \text{if $f$ is self-concordant}, \label{eq:lb_scf}
% % & \min\left\{\frac{1}{T^2},\frac{1}{\sqrt{N/m}},\frac{1}{\left(N/m\right)^{1/2}T^{3/2}}\right\} ~~ \text{ if $f$ is quasi-self-concordant}.   \label{eq:lb_qscf}
% \end{align}
\begin{equation}\label{eq:lb_qscf}
    \begin{aligned}
     \min\left\{\frac{1}{T^2},\frac{1}{\sqrt{N/m}},\frac{1}{\left(N/m\right)^{1/2}T^{3/2}}\right\}\\
     \text{ if $f$ is quasi-self-concordant}.   
    \end{aligned}
\end{equation}

% \textcolor{orange}{Maybe I will remove the next sentence, it is a detail and is not SA, which is what we are discussing in this subsection.} \pd{Pavel:On the one hand I agree. But, on the other hand, this is a kind of motivation for us and for the next part on SAA.} The authors of \cite{woodworth2021minmax} also discuss a potential of achieving a smaller error by considering the SAA approach since in this approach it is allowed to access  
% each realization of random function $f(x,\xi)$ for fixed $\xi$ and different values of $x$, which is a more restrictive assumption than the one under which the lower bound \eqref{Wood2021} was obtained.
% \todo{Check that the following bound also for federated stochastic optimization}
The authors of \cite{godichon2020rates}, under an additional assumption of stronger local smoothness of $f$ around $x^*$, propose an algorithm with only one round of communication ($T = 1$) with the following guarantee
%\footnote{Note, that from this bound we have that the first term equals the second one when $m\simeq \sqrt{N}$, where $N=mnT$. This means that we can fully parallel $N$ oracle calls on at most $m\simeq\sqrt{N}$ machines. The bound from \cite{godichon2020rates} was obtained under additional assumptions $T = 1$ and strong convexity of $f$ (so we properly correct the estimate).} 
\begin{equation}
\label{eq:godichon2020rates}
  \EE{\boldsymbol{F}(x^T)} - \boldsymbol{F}(x^*) \lesssim \frac{1}{N/m} + \frac{1}{\sqrt{N}},
\end{equation}
which, as we will see below, is similar to the SAA-based approach. 

% \ag{I guess it'd be better to add a definition of $x^*$...}
% \pd{Done above}

\subsubsection{Sample Average Approximation}

 The alternative  SAA approach \cite{shalev2009stochastic,daneshmand2021newton} is based on sampling in advance $N$ realizations of random function $f(x,\xi)$ and approximating the expectation in \eqref{StocProb} by a regularized finite-sum 
 %with regularization term with parameter $\mu \simeq 1/\sqrt{N}$ 
 \begin{equation}\label{MC}
      \min_{x\in\RR^d}F(x)=\frac{1}{N}\sum_{k=1}^N f(x,\xi^{k}) + \frac{\mu}{2}\|x - x_0 \|^2,%+ \frac{1}{\sqrt{mn}}\|x - x^0\|^2
 \end{equation}
 where $\|\cdot\|$ is the Euclidean norm.
If the regularization parameter $\mu \simeq 1/\sqrt{N}$, solving the problem $\eqref{MC}$ with sufficient accuracy, we obtain the solution  of $\eqref{StocProb}$ (see Sec. \ref{sec:2} for details). 
 %  or a problem with a similar sum, but with different strategy for choosing summation superscripts and regularization parameter, see  \cite{woodworth2018graph}.
%With such a change of objective, solving the problem $\eqref{MC}$ with increased accuracy, we obtain the solution  of $\eqref{StocProb}$ (see Sec. \ref{sec:2} for details). 
%  If the regularization parameter $\mu\simeq 1/\sqrt{\aa{N}} = \aa{1/\sqrt{mn}}$ \ag{I propose not to substitute here $N = mn$ since we typically don't use this relation in this subsection.} and problem \eqref{MC} is solved with accuracy $\varepsilon'=O(\mu\varepsilon^2)$, i.e. we find a point $x^T$ such that $\EE{F(x^T)} - F(x_F^*)\le \varepsilon'$, where $x_F^*$ is the solution to \eqref{MC}, then, for the original problem \eqref{StocProb}, we have $\EE{\boldsymbol{F}(x^T)} - \boldsymbol{F}(x^*)\le \varepsilon$ \cite{shalev2009stochastic}.
This motivates developing fast algorithms for problem \eqref{MC} with the ultimate goal being to obtain solution to the original stochastic optimization problem \eqref{StocProb}. 

For the SAA approach we assume that we have in total $N$  realizations $\{\xi^k\}_{k = 1}^N$ of the random variable $\xi$, $T$ communication rounds, and $m$ workers. Each worker can perform $n$ local steps between two communication rounds with each step using one gradient $f(x,\xi^{k})$ for a particular realization $\xi^{k}$. The difference with the SA approach is that it is possible to use the same realization $\xi^k$ in different local steps and also that, in general, $n\ne N/(mT)$.
 %So, the bounds which we write below are for an approximate solution to \eqref{StocProb}.} 
 %Note, the underlying assumption is that we have   access to $\nabla f(x,\xi)$ at several points $x$, when $\xi$ is chosen. 
 
 Although the SAA approach allows using each observation multiple times, there are Variance Reduced (VR) methods, which applied to the problem $\eqref{MC}$, typically use only a logarithmic number of  gradients for each observation $\xi^t$. 

% The best known result \aa{Fix}, when $m$ is large enough,  
%  %\ag{I'm not sure that we should use the word ``stochastic'' in this  context...} 
%  , which gives the following convergence rates 
 %(cf. \eqref{eq:godichon2020rates}) 
The following convergence rate for the original problem \eqref{StocProb} was obtained in \cite{woodworth2018graph} by using SVRG algorithm

 %\aa{depending on the relation between $n, m, T$} \todo{some parameters like $L_0$ are not defined}
%  %\todo{Just to be sure: these bounds are for the original stochastic problem \eqref{StocProb}, right?}
%  \begin{equation}\label{SVRG}
%       \EE{\boldsymbol{F}(x^T)} - \boldsymbol{F}(x^*) \le \frac{1}{\aa{N/m}} + \frac{1}{\aa{\sqrt{N}}},
%  \end{equation}
%  \aa{if $\dfrac{nTL_0^2}{L_1R^2} > m\log\left(\dfrac{nmT}{L} \right)$, and
%   \begin{equation}\label{SVRG2}
%       \EE{\boldsymbol{F}(x^T)} - \boldsymbol{F}(x^*) \le \aa{\frac{T}{\sqrt{N}} + \frac{\sqrt{T}}{\aa{\sqrt{mN}}}}.
%  \end{equation}
%  otherwise.}
\begin{equation}\label{SVRG}
    \EE{\boldsymbol{F}(x^T)} - \boldsymbol{F}(x^*) \lesssim  \exp \left(-\min\left\{\frac{nT}{\sqrt{N}}, \frac{mnT}{N}\right\}\right) + \frac{1}{\sqrt{N}}.
\end{equation}
% \ag{Note, \eqref{SVRG} assumes that since $nT$ is everywhere together in \eqref{SVRG}, we could communicate only one time $T=1$ at the end, see also \eqref{eq:godichon2020rates} that corresponds (up to a log-factors) to the same rate and communications requirements.}
%   Somewhere in between SA and SAA approach lie accelerated VR methods  \cite{li2020optimal} \ag{This is also true for SVRG  approach with \eqref{SVRG} convergence rate.} which actually use only a logarithmic number of stochastic gradients for each $\xi$ and allow to obtain the following bounds, which are no better than the above bound and the bound \eqref{eq:godichon2020rates} for the SA approach,
 %Accelerated versions VR methods \cite{lan2020first} lead to the rates which are no faster \cite{li2020optimal}:
 %The results \eqref{Wood2021} and \eqref{SVRG} describe the modern state of affairs.
 %For SAA approach we note, that described upper bound \eqref{SVRG} can not be improved by optimal VR-schemes for decentralized setup  (federated learning is a special case (centralized) of this general setup). 
 %From \cite{li2020optimal} we have only
Note, that parameters $n$ and $T$ appear in the inequality $\eqref{SVRG}$ only as a combination $nT$. Therefore, the minimum number of communication rounds $T$ is achieved when $T=1$. 
Further, from the bound $\eqref{SVRG}$ we see that when the number of workers $m$ increases, there is a limit for possible improvement in the bound.
Indeed, if $m\gtrsim \sqrt{N}$ the minimum in the exponent in \eqref{SVRG} is achieved on the first term, and no improvement in the bound is guaranteed. Moreover, the similar limit $m \simeq \sqrt{N}$ can be obtained via SA-based method from $\cite{godichon2020rates}$ with the guarantee $\eqref{eq:godichon2020rates}$.

%\aa{Note, that we are interested in minimizing the number of communication rounds $T$ and maximizing the number of workers $m$ at the same time. From the bound $\eqref{SVRG}$ we derive that $m \simeq \sqrt{N}$. Indeed, if $m\gtrsim \sqrt{N}$ that the minimum in $\eqref{SVRG}$ is achieved on the second argument. Therefore, subsequent increase in number of workers $m$ does not improve the convergence rate. The same result $m \simeq \sqrt{N}$ can be obtained from $\eqref{eq:godichon2020rates}$.}
 
 Accelerated (non distributed) Variance Reduced schemes  \cite{lan2020first} can not improve  the above bound \eqref{SVRG}. 
Non-accelerated distributed   Variance Reduced method from \cite{gorbunov2020local} applied to the special type of the problem $\eqref{MC}$ 
   \begin{equation}\label{eq:MC1}
      \min_{x\in\RR^d}F(x)=\frac{1}{m}\sum_{i=1}^m\frac{1}{K}\sum_{j=1}^K f(x,\xi^{i,j}) + \frac{\mu}{2}\|x - x_0\|^2,%+ \frac{1}{\sqrt{mn}}\|x - x^0\|^2
 \end{equation}
 where $K = N/m$, gives only the following bound for an approximate solution to problem \eqref{StocProb} (see Appendix \ref{S:SOTA} for details):
 \begin{equation}\label{eq:nonacc_VR}
\begin{aligned}
    %  \text{(No acceleration \cite{gorbunov2020local})} ~~
     \EE{\boldsymbol{F}(x^T)} - \boldsymbol{F}(x^*) \lesssim& \\ 
     \exp\left( - \min{\left\{\frac{T}{\sqrt{N}}, \frac{mnT}{N}\right\}} \right) 
     &+ \frac{1}{\sqrt{N}}.
\end{aligned}
\end{equation}

The RHS of the above inequality consists of two terms. The first one (optimization error) corresponds to the inexact solution of the approximation $\eqref{eq:MC1}$ and the second one (statistical error) comes from statistical reasoning of how well  \eqref{eq:MC1} approximates \eqref{StocProb}. 
Ideally, optimization error and statistical error should be of the same order. Indeed, if, due to a small budget of communications, the optimization error dominates, the sample size $N$ should have been chosen smaller. On the other hand, there is no much sense in optimizing below the statistical error.
% If the considered method performs not enough communication rounds, the first term will be larger than the second one and the convergence rate will mainly depend on it. On the other hand, when the number of communications is too large, the second term majorizes the first one. Since we want 
Thus, to minimize $T$, we are interested in the regime when optimization and statistical errors are of the same order. This is achieved when $T \simeq \max \{N^{1/2}, N/(mn)\}$ (recall that $\simeq$ hides also logarithmic factors). At the same time, we would like to maximize the number of workers to scale up the computations. If $m$ is too large, i.e. $m \gtrsim N^{1/2}$, the minimum in the exponent is achieved at the first term and there is no improvement in convergence rate with the increase of $m$. Therefore, the best possible choice is $m \simeq \sqrt{N}$ (where we set $n = 1$ to maximize $m$) and $T \simeq \sqrt{N}$. In the rest of the paper, we follow the same scheme to estimate a sufficient number of communication rounds and the possible numbers of workers. 
 
Optimal (accelerated) distributed   Variance Reduced method from \cite{li2020optimal} gives
\begin{equation}\label{eq:acc_VR}
\begin{aligned}
    % \text{(Acceleration)} ~~ 
    \EE{\boldsymbol{F}(x^T)} - \boldsymbol{F}(x^*) \lesssim & \\ 
     \exp\left( -\min{\left\{\frac{T}{N^{1/4}}, \frac{mnT}{N}, \frac{\sqrt{m}nT}{N^{3/4}}\right\}} \right)& 
     +  \frac{1}{\sqrt{N}}.
\end{aligned}
\end{equation}

From $\eqref{eq:acc_VR}$ following the same reasoning as before we derive that to minimize $T$ and $m$, we should choose $T \simeq N^{1/4}$ and $m \simeq N$. Interestingly, the same result is achieved by the standard accelerated gradient method \cite{nesterov2018lectures} applied to the finite-sum problem with $N$ terms in the sum.
%for the finite-sum problem (with $N$ terms) has the same result: $T \simeq N^{1/4}$ and $m \simeq N$.}
% communication requirements. }

%  $$\EE{\boldsymbol{F}(x^T)} - \boldsymbol{F}(x^*) \le \frac{1}{T} + \frac{1}{{\sqrt{mnT}}} ~~\text{(no acceleration \cite{gorbunov2020local})}.$$

%\todo{the division between statistical similarity subsubsection and SAA subsubsection is vague. Should be clarified.}

%With additional assumption about $f$ (i.e. $f(x,\xi) := f\left(\langle a(\xi), x \rangle\right)$) we can additionally adjust second-order (centralized) distributed Newton-type methods for \eqref{MC} to obtain better estimates \cite{islamov2021distributed}. Note, that our result \eqref{ANPN} does not assume $f(x,\xi) := f\left(\langle a(\xi), x \rangle\right)$.

\subsubsection{Exploiting statistical similarity} 
Recent advances in distributed optimization for solving problem \eqref{MC} are achieved by distributing $N$ realizations of $f(x,\xi)$ between $m$ workers each having $n=N/m$ realizations.
Then problem \eqref{MC} takes the form 
\begin{equation}\label{MC1}
  \min_{x\in\RR^d}F(x)=\frac{1}{m}\sum_{k=1}^m f_k(x)+ \frac{\mu}{2}\|x - x_0\|^2,
\end{equation}
where $\mu \simeq 1/\sqrt{N}=1/\sqrt{mn}$, $f_k(x) = \frac{1}{n}\sum_{j=1}^n f(x,\xi^{k,j})$. Using probabilistic arguments statistical similarity is shown between $f_k$ and the whole sum.
More formally, $\|\nabla^2 f_k(x) -  \frac{1}{m}\sum_{i=1}^m\nabla^2 f_i(x)\|\le\beta$, where $\beta \simeq 1/\sqrt{n}$. 
%In many machine learning problems, by standard statistical reasoning, each local dataset in a computational network, is a good approximation to the whole global dataset, meaning that in some sense local loss on a machine is statistically similar to the global loss summed up over all the machines. 
This idea have been recently extensively exploited for optimization problems (mainly) over master/workers architectures,  under the name of statistical preconditioning \cite{pmlr-v32-shamir14,reddi2016aide,JMLR:v21:19-764,pmlr-v119-hendrikx20a,dragomir2019optimal,sun2019convergence,hendrikx2020statistically,dvurechensky2021hyperfast}. 
These papers focus on solving the finite-sum problem \eqref{MC} and most of them do not achieve the lower communication complexity bound for this setting obtained in \cite{arjevani2015communication}:\footnote{Here and below we use $O()$, $\Omega()$ notation to denote bounds which hold up to constant factors, and $\widetilde O()$, $\widetilde\Omega()$ notation to denote bounds which hold up to constant and polylogarithmic factors.} 
\begin{equation}\label{LB}
\Omega\left(\sqrt{1+ \frac{\beta}{\mu}}\ln\left(\frac{\Delta F(x^0)}{ \varepsilon'}\right)\right)
\end{equation}
to solve problem \eqref{MC} with accuracy $\varepsilon'$, where $\Delta F(x^0):=F(x^0)-F(x_F^*)$ and $x_F^*$ is the solution to this problem.
The authors of \cite{Zhang2018} propose a distributed implementation of the damped Newton Method called DISCO in the master/workers architecture for minimizing $M$-self-concordant functions and achieve the complexity bound
\begin{equation}\label{DISCO}
O\left(\left(M^2\Delta F(x^0) + \ln \frac{1}{\varepsilon'}\right)\sqrt{1+ \frac{\beta}{\mu}}\right).
\end{equation}
%\textcolor{red}{DISCO does!}
%Currently it is an open question how to reach this lower bound by using only first-order oracles at each node in the network. The results of \cite{hendrikx2020statistically} (in centralized case) and negative results of \cite{dragomir2019optimal} lead to a conjecture that this lower bound \eqref{LB} hardly can be achieved based only on first-order methods.\textcolor{red}{my understanding is that results there are under a different oracle and assumptions}
Unlike \cite{Zhang2018}, we propose to reach \eqref{LB} by using cubic-regularized Newton's method at the central node.
%Since the size of the message between nodes remains $\mathcal{O}(d)$ as for first-order methods, our approach allows to reduce communication complexity without additional communication overhead compared to  first-order methods.
%relaxation seems to be an important also from the practical point of view. 
%In our approach we consider only ANPN method (second-order method). 
Moreover, most of the above statistical preconditioning papers focus on \eqref{MC} and do not account for the actual goal of solving the original stochastic optimization problem \eqref{StocProb}.
Under an appropriate choice of the parameter $\beta \simeq 1/\sqrt{n} = 1/\sqrt{N/m}$, when combined with statistical reasoning, the best result in the literature for \eqref{StocProb} corresponds to the guarantee 
% \ag{I guess it'd be better to rewrite the next two formulas in the same way as \eqref{SVRG} (in this case it'd be better to add another reformulation of \eqref{LB}) or in the same way as \eqref{LB}.}
% $$
% \EE{\boldsymbol{F}(x^T)} - \boldsymbol{F}(x^*) \le \frac{\beta}{T^{\kappa(\boldsymbol{F})}} + \frac{1}{\sqrt{\aa{N}}} =  \frac{1}{\aa{(N/m)}^{1/2}T^{\kappa(\boldsymbol{F})}} + \frac{1}{\sqrt{\aa{N}}},
% $$

\begin{equation}\label{eq:bound_kappa}
     \EE \boldsymbol{F}(x^T) - \boldsymbol{F}(x^*) \lesssim 
     \exp\left(-\frac{T}{m^{1/(2\kappa(\boldsymbol{F}))}}\right) + \frac{1}{\sqrt{N}},
\end{equation}
where $\kappa(\boldsymbol{F})\in [1,2]$ and, in general \cite{dragomir2019optimal}, $\kappa(\boldsymbol{F}) \simeq 1$.
Moreover, the methods achieving this bound \cite{sun2019convergence,hendrikx2020statistically,dvurechensky2021hyperfast} require to solve rather difficult auxiliary problems at each node. 

In \cite{daneshmand2021newton} the latter drawback is overcome by applying \textit{non-accelerated} cubic regularized Newton step \cite{nesterov2006cubic} at each node in order to solve \eqref{MC}, which results in the bound
% $$
% \EE{\boldsymbol{F}(x^T)} - \boldsymbol{F}(x^*) \le \frac{1}{T^2}  + \frac{\beta}{T} + \frac{1}{\sqrt{\aa{N}}} =  \frac{1}{T^2} + \frac{1}{\aa{(N/m)}^{1/2}T} + \frac{1}{\sqrt{\aa{N}}}
% $$

\begin{equation}\label{eq:Cubic_conv_notacc}
    \begin{aligned}
         \EE \boldsymbol{F}(x^T) - \boldsymbol{F}(x^*) \lesssim& \\
     \exp\left(-\min \left\{ \frac{T}{N^{1/4}}, \frac{T}{m^{1/2}} \right\}\right)&+ \frac{1}{\sqrt{N}}
    \end{aligned}
\end{equation}
for the problem \eqref{StocProb}.

 \subsection{Our contribution}
 The main contribution of this paper is two-fold. First, we focus on  finite-sum problems under statistical similarity and propose a master/workers distributed algorithm for such problems. 
 The main idea is to implement inexact accelerated cubic regularized Newton's method \cite{nesterov2008accelerating,baes2009estimate,ghadimi2017second-order} at the master node for functions with $L$-Lipschitz Hessian, which allows to obtain communication complexity bound 
 \begin{equation}\label{Our_bound}
O\left(\sqrt{\frac{\beta}{\mu}} \ln \frac{1}{\varepsilon'}  +  \left( \frac{L^2\Delta F(x_0)}{\mu^{3}} \right)^{1/6}\right)
\end{equation}
 that is better than the bound \eqref{DISCO} in \cite{Zhang2018} since $M=L/\mu^{3/2}$, and matches the dependence on $\beta$ and $\mu$ in the lower bound \eqref{LB}. Since the size of the message between nodes remains $\mathcal{O}(d)$ as for first-order methods, our approach allows to reduce communication complexity without additional communication overhead compared to  first-order methods.
 %The main idea is, unlike previous works \textcolor{red}{[what about DISCO?]}, to go beyond first-order methods and allow second-order steps on the central node in order to implement inexact accelerated cubic regularized Newton's method \cite{nesterov2008accelerating,baes2009estimate,ghadimi2017second-order}.
 %This allows us to obtain a centralized distributed algorithm with communication complexity matching the lower bound  \cite{arjevani2015communication} specialized for the centralized setup and given by \eqref{LB}.
 
% \ag{I propose to delete the next paragraph. This is not correct (I mean formula \eqref{ANPN}) and to formulate the main result \eqref{ANPN} like \eqref{SVRG}:
% \begin{gather*}
%      \boldsymbol{F}(x^T) - \boldsymbol{F}(x^*) \lesssim \\
%      \exp\left(-\frac{T}{N^{1/6}}\right) + \exp\left(-\frac{T}{(N/m)^{1/4}}\right) + \frac{1}{\sqrt{N}}.
% \end{gather*}
% }
 Second, we apply this method in order to solve the original stochastic optimization problem \eqref{StocProb} and obtain an algorithm that converges according to the following bound
 %In this paper (assuming additionally second-order smoothness of $f$) we develop distributed version of Accelerated Nesterov--Polyak--Newton (ANPN) method (Accelerated Cubic Regularized Newton method \cite{nesterov2008accelerating}) that converges according to the following estimate
%  \begin{equation}\label{ANPN}
%       \EE{\boldsymbol{F}(x^T)} - \boldsymbol{F}(x^*) \le   \frac{1}{T^3} + \frac{1}{\aa{(N/m)}^{1/2}T^2} + \frac{1}{\aa{\sqrt{N}}}.
%  \end{equation}
\begin{equation}\label{eq:Cubic_conv}
    \begin{aligned}
         \EE \boldsymbol{F}(x^T) - \boldsymbol{F}(x^*) \lesssim ~& \\
         \exp\left(-\min \left\{\frac{T}{N^{1/6}}, \frac{T}{m^{1/4}}\right\}\right)& + \frac{1}{\sqrt{N}}.
    \end{aligned}
\end{equation}

%  This bound is better than \eqref{Wood2021}, \eqref{SVRG}, in particular, when $\pd{N/m}=n \leq c_1 T^2 \leq c_2 m^{1/2}$ for some small numerical constants $c_1,c_2$
 %Note, that only in self-concordant cases we have that corresponding lower bounds are strictly better than our upper bound \eqref{ANPN}.

%  Note that for brevity  we consider only non-strongly convex case (therefore we used regularization in \eqref{MC} and \eqref{sum}).  
%  \todo{write about strongly convex case}
 
Under additional assumption of $\mu$-strong convexity of the original problem $\eqref{StocProb}$, the proposed algorithm provides the bound
\begin{gather*}
     \EE \boldsymbol{F}(x^T) - \boldsymbol{F}(x^*) \lesssim \\
     \exp\left(- \min\left\{T\mu ^{1/3}, T\frac{\mu^{1/2}N^{1/4}}{m^{1/4}}\right\}\right) + \frac{1}{\mu N}.
\end{gather*}

In Table \ref{tab:comp} we summarize a comparison with the related works described above. Note, that in our work we are motivated by two goals: minimize the number of communications $T$ and maximize the number of workers $m$ to achieve better parallelization possibilities. 
\begin{table}[ht]
\caption{Comparison between different methods and bounds for problem \eqref{StocProb}}\label{tab:comp}
    \begin{center}
        \begin{tabular}{|c|c|c|c|}
            \hline 
             & Bound & $T$ & $m$\tabularnewline
            \hline 
            SA lower bound \cite{woodworth2021minmax}& $\eqref{Wood2021}$ & $N^{1/4}$ & $N^{3/4}$\tabularnewline
            \hline 
            SGD \cite{godichon2020rates}&$\eqref{eq:godichon2020rates}$ & $1$ & $N^{1/2}$\tabularnewline
            \hline 
            SVRG \cite{woodworth2018graph}&$\eqref{SVRG}$  & $1$ & $N^{1/2}$\tabularnewline
            \hline 
            Non-accelerated VR \cite{gorbunov2020local}& $\eqref{eq:nonacc_VR}$  & $N^{1/2}$ & $N^{1/2}$ \tabularnewline
            \hline 
            Accelerated VR \cite{li2020optimal}& $\eqref{eq:acc_VR}$  & $N^{1/4}$ & $N$ \tabularnewline
            \hline 
            Cubic Newton \cite{daneshmand2021newton} & $\eqref{eq:Cubic_conv_notacc}$   & $N^{1/4}$ & $N^{1/2}$ \tabularnewline
            \hline 
           \begin{tabular}{@{}c@{}}Accelerated Cubic Newton \\  \text{[this paper]} \end{tabular} & $\eqref{eq:Cubic_conv}$ & $N^{1/6}$ & $N^{2/3}$ \tabularnewline
            \hline 
        \end{tabular}
    \end{center}
\end{table}

\section{Finite-Sum Approximation for Stochastic Optimization Problem }\label{sec:2}

%\todo{State the main assumptions about the stochastic optimization problem \eqref{StocProb}, write its regularized finite-sum counterpart which we will solve. Write explicitly the results on how these two problems are connected, how to choose the regularization parameter, accuracy of the solution to the finite-sum optimization problem (and the criterion in which this accuracy is defined) which guarantees the desired accuracy of the solution to the original stochastic optimization problem. State the main properties of the finite-sum problem such as statistical similarity. What is not clear so far is how the regularizer enters the similarity: statistical similarity is between the small sum and large sum, without regularizer, as it seems from the SPAG paper. Also everything should be done for the another case of strongly convex function in \eqref{StocProb}.}
%\todo{Add the case of strongly convex function in \eqref{StocProb}.}

%\todo{Write the plan of the section so that the reader knows what to expect.}
In this section, we state the main assumptions about stochastic optimization problem \eqref{StocProb}, discuss its finite-sum approximation obtained via the SAA approach, as well as introduce and motivate the concept of statistical similarity for finite-sum optimization problems.

We consider stochastic minimization problem
\begin{equation}\label{eq:stoch_problem}
    \min_{x\in\RR^d}\boldsymbol{F}(x):=\EE_{\xi}{f(x,\xi)},
\end{equation}
where $\xi$ is a random variable, $f(x, \xi)$ is convex function w.r.t. $x \in \RR^d$ for all $\xi$. We assume that, for all $\xi$, $f(x, \xi)$ is $L_0$-Lipschitz continuous, i.e.,  
%\ag{This is not correct definition of Lipschitz continuity. It's Lipschitz continuity of gradient...} \pd{Corrected}
\begin{equation*}
    \| f(x, \xi) - \  f(y, \xi)\| \leq L_0 \|x - y\|~ \forall x, y \in \RR^d
\end{equation*}
and has $L$-Lipschitz Hessian, i.e.
\begin{equation*}
    \|\nabla^2 f(x, \xi) - \nabla^2 f(y, \xi)\| \leq L \|x - y\|~ \forall x, y \in \RR^d.
\end{equation*}
% \todo{Please unify notation for the norm. It is $\|\|$ in the intro and $\|\|$ here and below. Also it should be explicitly written that %the normis Euclidean.}
%\todo{I think we also need to explicitly write what we mean by $O$ and $\tilde{O}$.}
Note that the first of these assumptions implies that $\boldsymbol{F}$ is $L_0$-Lipschitz continuous. 
We will also consider the case when $f(x,\xi)$ is $\mu$-strongly convex for all $\xi$ as functions of $x$.

% \textcolor{orange}{I would write formally all the assumptions on the stochastic and the SAA problem as we introduce the problem}
% \pd{Pavel: I agree about the assumptions on the stochastic problem. In particular, later we use that $f(x, \xi)$ is $L_0$-Lipschitz. As far as I understand, the assumptions on the SAA problem follow from the assumptions in the stochastic problem.}

\subsection{Finite-sum approximation}
\label{S:finite_sum_approx}
To solve  problem $\eqref{eq:stoch_problem}$ we apply the SAA approach, i.e. sample $N$ iid realizations $\xi^{k,j}$ from an unknown distribution of $\xi$, for each worker node $k=1,...,m$ define local objective $f_k(x) = \frac{1}{n}\sum_{j=1}^n f(x,\xi^{k,j})$, where $n=N/m$, and approximate the original stochastic optimization problem \eqref{eq:stoch_problem} by the finite-sum problem
 \begin{equation}\label{MC1}
      \min_{x\in\RR^d}F(x):=\frac{1}{m}\sum_{k=1}^m f_k(x)
      %\frac{1}{n}\sum_{j=1}^n f(x,\xi^{k,j}) 
      + \frac{\mu}{2}\|x - x_0\|^2.
 \end{equation}
 %where $mn \simeq \frac{1}{\e^2}$.
 %The regularization with parameter \aa{$\mu = O \left(\frac{\e}{R^2}\right)$} in $\eqref{MC1}$, where $\|x^* - x^0\| \leq R$ and $ \boldsymbol{F}(x_T) - \boldsymbol{F}(x^*) \leq \e,$ is crucial to guarantee that the
%minimizer of $\eqref{MC1}$ is a good approximation to the stochastic optimization problem $\eqref{eq:stoch_problem}$ \cite{shalev2009stochastic}.
%But a regularized empirical minimization problem $\eqref{MC1}$ has to be solved with higher accuracy. Indeed, 
We use $ x^*,  x_F^*$ to denote the solutions of problems $\eqref{eq:stoch_problem}$ and  $\eqref{MC1}$ respectively.  We assume that $x^*$ and $x_F^*$ lie inside the Euclidean ball with center at $x_0$ and radius $R$.

% Theorem 6 from \cite{shalev2009stochastic} claims that  under the assumption of  $L_0$-Lipschitz continuity of $f(x, \xi)$ w.r.t $x$ and with $\mu = \sqrt {\frac{L_0^2}{\delta N R^2}} = \sqrt {\frac{L_0^2}{\delta mn R^2}}$  the following bound holds  with probability at least $1 - \delta$:
Corollary 1.2 from \cite{feldman2019high} claims that  under the assumption of  $L_0$-Lipschitz continuity of $f(x, \xi)$ w.r.t $x$ and with $\mu =  {\frac{L_0\log N}{RN}}$
% = {\frac{L_0\log (mn)}{Rmn}}$ 
the following bound holds  with probability at least $1 - \delta$:
%\ag{I'm not sure about $L_0$. Above we determine $L_0$ as Lipschits-gradient constant. May be we should redefine $L_0$ as Lipschits constant? What is $\delta$ (here and below)? Maybe it's worth to include with probability greater than $1 - \delta$?}
 $$
 \boldsymbol{F}(x_F^*) - \boldsymbol{F}(x^*) \leq O\left({\frac{L_0 R}{\sqrt{N}}}\log\left(N/\delta\right)\right).
 $$
 %\todo{How is $R$ defined here?}
 Using $L_0$-Lipschitz continuity of $\boldsymbol{F}(x)$, we obtain that
 $$
 \boldsymbol{F}(x) - \boldsymbol{F}(x_F^*) \leq L_0 \|x - x_F^*\|, \; \forall x.
 $$
 Combining the above two inequalities and plugging $x = x^T$, where $x^T$ is an output of some optimization algorithm after $T$ communication rounds, we obtain
 \begin{equation}
 \label{eq:bound_fin_sum_to_bound_stoch}
 \boldsymbol{F}(x^T) - \boldsymbol{F}(x^*) \leq   L_0 \|x^T - x_F^*\| +
 % O\left(\sqrt{\frac{L_0^2 R^2}{\delta N}}\right)
 O\left({\frac{L_0 R}{\sqrt{N}}}\log\left(N/\delta\right)\right).
 \end{equation}
 Thus, if we find a good approximation $x^T$ to the solution $x_F^*$ of the finite-sum problem \eqref{MC1}, we automatically obtain an approximate solution to problem  \eqref{eq:stoch_problem}.
 
If we additionally assume that $f(x, \xi)$ is $\mu$-strongly convex for all $\xi$, there is no need in additional regularization and the original stochastic problem $\eqref{eq:stoch_problem}$ can be approximated by 
 $$
 \min \limits_x \frac{1}{m}\sum_{k=1}^m f_k(x).
 $$
 In that case the bound $\eqref{eq:bound_fin_sum_to_bound_stoch}$ can be improved \cite{shalev2009stochastic} to:
 \begin{equation}\label{eq:bound_strong}
       \EE \boldsymbol{F}(x^T) - \boldsymbol{F}(x^*) \leq   L_0 \|x^T - x_F^*\| + O\left({\frac{L_0^2}{ \mu N}}\right).
 \end{equation}

%  Using $L_0$-Lipschitz continuity of $\boldsymbol{F}(x)$ and the strong convexity of $F$ for any $x \in \mathbb{R}^d$ we get
%  $$
%  \boldsymbol{F}(x) - \boldsymbol{F}(\aa{x_F^*}) \leq \sqrt{\frac{2L_0^2}{\mu}}\sqrt{F(x) - F(\aa{x_F^*})}.
%  $$
%  Combining two inequalities from above and plugging $x:= \aa{x^T}$, where $\aa{x^T}$ is an output of some optimization algorithm we obtain
%  $$
%  \boldsymbol{F}(\aa{x^T}) - \boldsymbol{F}(\aa{x^*}) \leq O\left(\sqrt{\frac{L_0^2 R^2}{\delta mn}}\right) + \sqrt{\frac{2L_0^2}{\mu}}\sqrt{F(\aa{x^T}) - F(\aa{x_F^*})}.
%  $$

% Which means that if we want to solve an original stochastic problem $\eqref{eq:stoch_problem}$ with accuracy $\e$ \aa{we need to choose $mn = O\left( \frac{L_0^2R^2}{\e^2} \right)$ and solve the finite-sum problem $\eqref{MC1}$ with accuracy $\e' = \frac{\e^2\mu}{2L_0^2} = O\left(\frac{\e^3}{L_0^2 R^2}\right)$.}
% \todo{It is better to define in which sense we mean accuracy of the solution.} \textcolor{orange}{I agree!}

% \todo{We need to think about different notation for $Ef(x,\xi)$ in \eqref{StocProb} and its finite-sum approximation in \eqref{MC1}. I used $\boldsymbol{F}$ in \eqref{StocProb} and $F$ in \eqref{MC1}. But, maybe there is a better option. Also it is not clear whether we should define $F$ to be regularized finite-sum or just the finite sum.}
%We assume that $$\|\nabla^2 f_k(x) - \nabla^2 f(x)\|\le\beta.$$ 
% For \eqref{MC} we have $f_k(x) = \frac{1}{n}\sum_{j=1}^n f(x,\xi^{k,j})$, $\mu \simeq 1/\sqrt{mn}$ and $\beta \simeq 1/\sqrt{n}$ in \eqref{sum}.

\subsection{Statistical similarity}
\label{S:stat_similarity}
Since problem \eqref{MC1} originates from problem \eqref{eq:stoch_problem}, we can state and utilize one more important property of the objective function in \eqref{MC1}, namely, statistical similarity.
Under assumption that the observations $\xi^{k, j}$ are iid, the following bound holds \cite{zhang2018distributed} for all $k=1,...,m$ with  probability  at least $1 - \delta$: 
%\todo{a bit strange that $\delta$ does not enter the bound. How $r$ is defined? Also I'm confused because $F$ contains the regularizer but $f_k$ not.}\aa{ $\delta$ is in the logarithmic term}
\begin{equation}
\label{eq:stat_sim}
    \sup _{w}\left\|\frac{1}{m}\sum_{j=1}^m \nabla^2 f_j(w)-\nabla^2f_k(w)\right\| \leq \tilde{O}\left({\sqrt{\frac{32 L^{2} d}{n}}}\right).
    %\\  \min \left\{L, \sqrt{\frac{32 L^{2} d}{n}} \cdot \sqrt{\log \left(1+\frac{r M \sqrt{2 n}}{L}\right)+\frac{\log (m d / \delta)}{d}}\right\}.
\end{equation}
In the next section we utilize statistical similarity to propose an efficient distributed algorithm for problem \eqref{MC1}.

\section{Achieving the lower bound for finite-sum optimization under statistical similarity}
Motivated by the connection between the finite sum problem \eqref{MC1} and the original stochastic optimization problem \eqref{eq:stoch_problem} stated in the previous section, we propose in this section a distributed minimization algorithm with master/workers architecture for general finite-sum problems, in particular, 
problem \eqref{MC1}. We also show that this algorithm achieves the lower communication complexity bound in  \cite{arjevani2015communication} specialized for our setting of master/workers architecture. Moreover, our algorithm achieves communication complexity bound that is better than the one in \cite{Zhang2018}.
%for a more challenging setting of decentralized distributed optimization.

To that end we consider a network with $m$ agents and the following general finite-sum optimization problem
%Let us consider smooth convex optimization problem ($m$ -- the number of nodes)
%\todo{Actually, I've realized that in this section we can consider a more general problem, maybe without regularization, since afterwards we will apply this general algorithm also for the setting when \eqref{StocProb} is strongly convex. then regularization is not needed.}
\begin{equation}\label{sum1}
     \min_{x\in\RR^d}F(x):=\frac{1}{m}\sum_{k=1}^m f_k(x),% + \frac{\mu}{2}\|x - x^0\|^2,
\end{equation}
where each worker node $k$ has access only to its local part $f_k$ of the objective.
Note that this problem statement covers problem \eqref{MC1} as a special case since the regularizer can be equally distributed among the agents. We assume that $F$ is $\mu$-strongly convex and has $L$-Lispschitz Hessian.
% \textcolor{orange}{I think that there are repetitions with the previous section. I suggest to keep all the problem formulation, network description in the previous section, here directly the algorithm}
% \pd{Pavel: My idea was the following. From the previous section we see that we can solve the stochastic optimization if we solve finite-sum problem. In this section we focus on a general finite-sum problem which may be not related with the original stochastic problem. So, the idea was to present in this section a general distributed algorithm for finite-sum problems.}\textcolor{orange}{ok. }
Motivated by subsection \ref{S:stat_similarity}, we make in this section the following assumption that each local objective $f_k$ is a good approximation to the global objective $F$.
{\begin{assumption}\label{as:similarity}(statistical similarity)
    Each local function $f_k$ is $\beta$ related to the global objective $F$:
    \begin{equation}\label{eq:beta-sim}
    \|\nabla^2 F(x) - \nabla^2 f_k(x)\| \leq \beta,
    \end{equation}
\end{assumption}
for all $x \in \mathbb R^d$ and some $\beta > 0$.}
In particular, for problem \eqref{MC1} we have that $\beta = \tilde{O}(\sqrt{d/n})$ with high probability.
%{We also assume that the local empirical losses are statistically similar to the total loss. In other words, $f_k(x)$ is a good approximation to $\hat{F} = \sum \limits_{k = 1}^m f_k(x)$. Therefore, we have the following assumption}
%Which implies that $\beta$ from Assumption $\ref{as:similarity}$ is equal to the RHS of the inequality above and $\|\nabla^2 f(w) - \nabla^2 \phi(w)\| \sim \sqrt{\frac{d}{n}}$ with high probability.

%\todo{Specify the algorithm for our setting as much as it is possible: in particular the value of $\tau$ can be substituted to the algorithm. Explain what happens in the algorithm also in words. State the main theorem for our setting of statistical similarity. Give the necessary details to use the method of Ghadimi et al in the proof of the main theorem for our setting.}

To solve problem \eqref{sum1} we propose Restarted Distributed Accelerated Cubic Regularized Newton's Method by extending the methods of \cite{nesterov2008accelerating,baes2009estimate,ghadimi2017second-order}. 
First, we describe Distributed Accelerated  Cubic Regularized Newton's Method for minimizing convex functions (Algorithm \ref{alg:ANPN}). Then we apply restart technique to obtain linearly convergent Algorithm \ref{alg:restarts} for minimizing $\mu$-strongly convex function in \eqref{sum1}.

We choose one of the agents (w.l.o.g. the agent with number 1) to be the central node (server), and all the others are $m-1$ to be workers (machines) that are assumed to be connected with the central node. 
Each one of these nodes stores the part $f_k$ of the global objective, computes gradients, and passes them to the central node.  Then, the server forms the gradient of the global objective $\nabla F(x)$, computes the Hessian of its local loss $f_1(x)$, constructs the following model of the global objective $F$:
\begin{align}
&\widetilde{{F}}_M(x, z) =F\left(z\right)+\left\langle\nabla F\left(z\right), x-z\right\rangle \\
&+\frac{1}{2}\left\langle (\nabla^2 {f_1}(z) +  {3\beta}I)\left(x-z\right), x-z\right\rangle+\frac{M}{6}\left\|x-z\right\|^{3},
\end{align}
updates variable $x$ by minimizing this model and broadcasts it to the workers.

Algorithm \ref{alg:ANPN} is a master/workers generalization of the Accelerated Cubic Newton method under inexact second-order information \cite{ghadimi2017second-order}. We use local objective $f_1(x)$  on the server as an approximation to the global objective $F$. Note, that we need to compute $\nabla^2 f_1(x)$ only once per iteration, at the point $w_t$, but at the same time we need two communication rounds per one iteration. Importantly, each communication round requires sending only vectors. Also, our approach allows to vary sample size on the central node. This can lead to the better performance, since it improves the constant $\beta$ obtained from statistical similarity.

\begin{algorithm}[t]
    \caption{Accelerated Cubic Newton}\label{alg:ANPN}
	\textbf{Input}: $x_{0} \in \mathbb{R}^d$, $\alpha_0 = 1$, $A_0=1$, $L$, $\beta$.
	%, $\{\alpha_t\}_{t \geq 1} \in (0, 1)$, 
	%\aa{nonnegative nondecreasing sequence $\{\hat{\lambda}_t\}_{t \geq 0}$}, 
	%and strong convexity parameter $\mu \geq 0$. 

	\textbf{Step 0:} \begin{subequations}
		
% 	    \textit{At every worker node } 
		
% 		Compute $n$ gradients $\nabla f(x_0, \xi^{k, j})$, $j \in [1, n]$ and send them to the server.
			
% 		\textit{At central node ($k = 1$)} 
		
		{Master node computes $\nabla F(x_0)$ by collecting $\nabla f_k(x_0)$ from workers.}
		\begin{gather*}
			x_1 = \argmin_{x\in \mathbb{R}^d} \widetilde{{F}}_{4L}(x, x_0),\\
			y_1 = \argmin_{x \in \mathbb{R}^d}\{{\psi_1}(x) := F(x_1) + 8\beta\|x - x_0\|^2 + 16L\|x - x_0\|^3\}.
		\end{gather*}
		Set $w_0 = x_0$ and $ t = 1$.
		
	\textbf{Step 1:} 
	
% 		\textit{At every worker node } 
		
% 		Compute $n$ gradients $\nabla f(w_t, \xi^{k, j})$, $j \in [1, n]$ and send them to the server.
		
% 		\textit{At central node } 
        Master node computes $\nabla F(w_t)$ by collecting $\nabla f_k(w_t)$ from workers.

		Set 
		$$
		\begin{aligned}
			w_{t} &=\left(1-\frac{3}{t + 3}\right) x_{t}+\frac{3}{t + 3} y_{t} \\
			x_{t+1} &=\arg \min _{x \in \mathbb{R}^{n}} \widetilde{F}_{4L}\left(x , w_{t}\right).
		\end{aligned}
		$$
	\textbf{Step 2:} 
	
% 		\textit{At every worker node} 
		
% 		Compute $n$ gradients $\nabla f(x_{t+1}, \xi^{k, j})$, $j \in [1, n]$ and send them to the server.
		
% 		\textit{At central node } 
        Master node computes $\nabla F(x_{t+1})$ by collecting $\nabla f_k(x_{t+1})$ from workers.
        %{Compute $\nabla F(x_{t+1})$ (requires communication)}
		
		Define $A_{t}=A_{t-1}(1-3/(t + 3))$,
		\begin{align}
		    &\psi_{t+1}(x):={\psi_{t}(x)}+4\beta\|x-x_0\|^2 \\
		    &+\frac{3}{A_t(t + 3)} \left(F\left(x_{t+1}\right)+\left\langle\nabla F\left(x_{t+1}\right), x-x_{t+1}\right\rangle \right).
		\end{align}

		%Compute  $y_{t+1}$ as
		$$
		y_{t+1}=\arg \min _{x \in \mathbb{R}^{n}}\left\{\psi_{t+1}(x)\right\}.
		$$
		%where
		
% 		$$
%             A_{t}:=\left\{\begin{array}{ll}
%             1, & t=0, \\
%             \prod_{i=1}^{t}\left(1-\alpha_{t}\right), & t \geq 1
%             \end{array}\right.,
%         $$
% 		and
% 		$$
%     		\mathrm{lb}\left(x ; x_{t+1}\right)=F\left(x_{t+1}\right)+\left\langle\nabla f\left(x_{t+1}\right), x-x_{t+1}\right\rangle
% 		$$
		\textbf{Step 3:} Set $t \leftarrow t+1$ and go to step 1 .

	\end{subequations}
\end{algorithm}

The next Theorem gives convergence rate of Algorithm $\ref{alg:ANPN}$.
\begin{theorem}
\label{Th:ACRNM_conv_convex}
    Let Assumption $\ref{as:similarity}$ hold and $F$  be convex function with $L$-Lipschitz Hessian and defined in  \eqref{sum1}. Then after $t$ 
    iterations of Algorithm $\ref{alg:ANPN}$
    % \todo{I'm afraid that there will be a confusion with $T$ rounds of communications. but maybe it is fine since this section is independent.}\aa{Maybe we can replace "iteration" with "communication rounds" and multiply  the RHS of convergance rate by 2} iterations of Algorithm $\ref{alg:ANPN}$ with the following choice of parameters
    % \begin{equation}\label{eq:params}
    %     \alpha_t = \frac{3}{t + 3}, ~ \bar{\lambda}_t = 8\beta(t+2) ~\forall t > 0, ~C = 96L, \text{ and } M = 4L
    % \end{equation}
    we have
    \begin{gather}\label{eq:conv_rate_convex}
        F(x_t) - F(x_F^*) \leq \frac{98L\|x_0 - x_F^*\|^3}{t^3} + \frac{48\beta\|x_0-x_F^*\|^2}{t^2}
    \end{gather}
    where $x_F^*$ is a solution to \eqref{sum1}.
    %  is an upper bound for the RHS of 
    % $$
    %     \left\|y_{t}-\bar{x}_{t}\right\| \leq \frac{\alpha_{t} A_{t}^{-1}\left\|\nabla f\left(x_{t+1}\right)\right\|}{\aa{\frac \mu 4}+\bar{\lambda}_{t}}.
    % $$
\end{theorem}
Note that $T$ iterations correspond to $2T$ communication rounds since each iteration requires two communication rounds.

\begin{proof}
%\todo{make the proof more neat. Write the correspondence between our parameters and the parameters of Ghadimi paper. First, the model on the central node corresponds to inexact model of Ghadimi. Then our parameters correspond to such values of their parameters. Reference to their theorem. }
To prove the theorem we would like to apply Theorem 11 of \cite{ghadimi2017second-order}, which analyzes an accelerated cubic-regularized Newton's method with inexact Hessian. Thus, first we show that in our algorithm the central node indeed runs accelerated cubic-regularized Newton's method with inexact Hessian (Algorithm 2 of \cite{ghadimi2017second-order}). Their algorithm uses an approximation $H_t$ for the Hessian that satisfies (in their notation $\mu_u$ instead of $\lambda$)
\begin{equation}\label{eq:inexact_ghadimi}
    \frac{\lambda}{2} I \preccurlyeq H_t - \nabla^2 F(w) \preccurlyeq \lambda I.
\end{equation}
Our assumption of $\beta$-similarity \eqref{eq:beta-sim} allows to choose $H_t = \nabla^2 f_1(w) + 3\beta I$ to satisfy $\eqref{eq:inexact_ghadimi}$ with $\lambda = 4\beta$. Indeed, the equation $\eqref{eq:beta-sim}$ implies
\begin{gather*}
    \beta \geq \frac{\|(\nabla^2 F(x) - \nabla^2 f_1(x))x\|\|x\|}{\|x\|^2} \geq \\
    \frac{\langle x, (\nabla^2 f_1(x) - \nabla^2 F(x))x\rangle}{\|x\|^2}. 
\end{gather*}
Therefore, 
\begin{gather*}
    - \beta I  \preccurlyeq \nabla^2 f_1(x) - \nabla^2 F(x) \preccurlyeq \beta I. 
\end{gather*}
Adding $3\beta I$ to the inequality above we obtain $\eqref{eq:inexact_ghadimi}$ with $\lambda = 4\beta$. 
Thus, we see that the central node using its local Hessian is equivalent to using inexact Hessian of the global objective $F$.
Since the central node calculates the gradient of the global objective $F$ by communicating with the nodes, the central node indeed implements cubic steps with inexact Hessian.
  
Our choice of the algorithm parameters corresponds to the following choice of the parameters in Algorithm 2 of \cite{ghadimi2017second-order} stated in their Theorem 11: 
\begin{align*}\label{eq:params}
    &\alpha_t = \frac{3}{t + 3}, \mu_u=4\beta, ~ \bar{\mu}_t = 8\beta(t+2), \gamma=L, \beta = 96L, \\
    &\eta = M = 4L.
\end{align*}
%Note that they use $\gamma$ to denote the Lipshitz constant of the Hessian, 
Applying  Theorem 11 of \cite{ghadimi2017second-order} we obtain the statement of our theorem.
\end{proof}

\begin{algorithm}[t]
    \caption{Restarted Accelerated Cubic Newton}\label{alg:restarts}
	\textbf{Input}: $z_0  \in \mathbb{R}^n$, strong convexity parameter $\mu > 0$, Lipschitz constant for Hessian $L$, and $R_0 > 0$ such that $\|z_0 - x_F^*\|\leq R_0$. 
	\textbf{For $s = 1, 2, \ldots$:} 
	\begin{enumerate}
	    \item Set $x_0 = z_{s - 1}$ and $R_s= \frac{R_0}{2^s}$.
	    \item Run Algorithm \ref{alg:ANPN} 
	    %with parameters defined in $\eqref{eq:params}$ 
	    for $T_s$ iterations, where
	    \begin{equation}\label{eq:restarts_it}
	        t_s = 2\max\left\{ \left(\frac{196LR_{s-1}}{\mu}\right)^\frac{1}{3}, 2\left(\frac{24\beta}{\mu} \right)^\frac{1}{2}\right\}.
	    \end{equation}
	    \item Set $z_s = x_{t_s}$.
	\end{enumerate}
\end{algorithm}
Our next step is to restart Algorithm \ref{alg:ANPN} in order to exploit  strong  convexity of the objective and obtain linear convergence rate. 
In each step of Algorithm $\ref{alg:restarts}$ we run Distributed Accelerated Cubic Regularized Method for the number of iterations, defined in $\eqref{eq:restarts_it}$. Then,  we use its output as the initial point for the next run of Algorithm $\ref{alg:ANPN}$ with reset parameters and so on. 
%The convergence rate is obtained in $\cite{ghadimi2017second-order}$:
\begin{theorem}
\label{Th:ACRNM_conv_str_convex}
Let the assumptions of Theorem \ref{Th:ACRNM_conv_convex} hold and additionally $F$ be $\mu$-strongly convex. Let $R_0 > 0$ be such that $\|z_0 - x^*\|\leq R_0$ and $\{z_s\}_{s \geq 0}$ be generated by Algorithm $\ref{alg:restarts}$. Then for any $s \geq 0$ we have 
    \begin{align}
&\|z_{s}-x_F^*\| \leq R_02^{-s}, \label{eq:restart_conv_arg} \\
&F(z_s) - F(x_F^*) \leq \mu R_0^2\cdot2^{-2s-1} %\leq (F(x_0) - F(x_F^*))2^{-2s}. 
\label{eq:restart_conv_func}
    \end{align}
Moreover,  the total communication and oracle complexities are
\begin{equation}\label{eq:ACRNM_complexity}
    O\left( \sqrt{\frac{\beta}{\mu}}\log\frac{F(x_0) - F(x_F^*)}{\e} + \left(\frac{LR_0}{\mu}\right)^\frac{1}{3} \right).
\end{equation}

%\todo{Maybe move here the bound \eqref{eq:ACRNM_coverg_T} which depends on the number of communications.} 
\end{theorem}
\begin{proof}
We prove by induction that $\|z_s-x_F^*\|\leq 2^{-s}\|x_0-x_F^*\| \leq R_s = 2^{-s}R_0$. For $s=0$ this obviously holds. By strong convexity and inequality \eqref{eq:conv_rate_convex}, we obtain that  
\begin{gather*}
    \|z_{s+1}-x_F^*\|^2\leq \frac{2}{\mu}(F(x_{t_{s}+1}) - F(x_F^*)) \\
    \leq \frac{2}{\mu} \left(\frac{98L\|z_s - x_F^*\|^3}{t_{s+1}^3} + \frac{48\beta\|z_s-x_F^*\|^2}{t_{s+1}^2} \right)\\
    \leq \frac{2}{\mu} \left(\frac{98LR_s^3}{\left(2\left(\frac{196LR_{s}}{\mu}\right)^\frac{1}{3}\right)^3} + \frac{48\beta R_s^2}{\left(4\left(\frac{24\beta}{\mu} \right)^\frac{1}{2}\right)^2} \right) \leq \frac{R_s^2}{4} = R_{s+1}^2.
\end{gather*}
Thus, by induction, we obtain that \eqref{eq:restart_conv_arg}, \eqref{eq:restart_conv_func} hold.

Next we provide the corresponding complexity bounds. From the above induction bounds, we obtain that after $S$ restarts the total number of iterations of Algorithm \ref{Th:ACRNM_conv_convex}, each requiring one call of the second-order oracle, two calls of the first-order oracle and two communication rounds, is no greater than
\begin{gather*}
    \sum_{s=1}^{S} t_{s} \leq 2\left(\frac{196 L R_{0}}{\mu}\right)^{\frac{1}{3}} \sum_{s=1}^{S} 2^{\frac{1-s}{3}}+4 S\left(\frac{24 \beta}{\mu}\right)^{\frac{1}{2}} \leq \\
    8\left(\frac{392 L R_{0}}{\mu}\right)^{\frac{1}{3}}+4 \sqrt{\frac{24\beta}{\mu}} \log _{4}\left[\frac{F\left(x_{0}\right)-F\left(x_F^*\right)}{\varepsilon}\right].
\end{gather*}
Therefore, the total communication and oracle complexities are given by \eqref{eq:ACRNM_complexity}.
\end{proof}

Let us now translate the bound \eqref{eq:restart_conv_arg} to the language of the number of iterations of Algorithm \ref{Th:ACRNM_conv_convex} and the number of communication rounds. Let $T$ be an even number of communications, which means that we made $t=T/2$ iterations of Algorithm \ref{Th:ACRNM_conv_convex}. Let $\tau_1=2\left(\frac{196LR_{0}}{\mu}\right)^\frac{1}{3}$ and $\tau_2=4\left(\frac{24\beta}{\mu} \right)^\frac{1}{2}$. Then $t_s$ in Algorithm \ref{alg:restarts} satisfies $t_s \leq \max\{\tau_1,\tau_2\}$ and after  $T$ communication rounds, this algorithm makes $s \geq \lfloor\frac{T}{2\max\{\tau_1,\tau_2\}} \rfloor$  restarts and generates a point $x^T=z_{s}$ such that 
%\todo{check, maybe division by 2 is needed in the last bound}
\begin{equation}\label{eq:ACRNM_coverg_T}
    \|x^T-x_F^*\| \leq R_02^{-\lfloor\frac{T}{2\max\{\tau_1,\tau_2\}} \rfloor}\leq R_02^{- \frac{T}{2\max\{\tau_1,\tau_2\}}}.
\end{equation}

At this point it is convenient to compare the complexity bound \eqref{eq:ACRNM_complexity} with the bounds in the literature. 
Firstly, in terms of the dependence on $\beta$ and $\mu$ our algorithm achieves the lower bound \eqref{LB} obtained in  \cite{arjevani2015communication}. Secondly, we compare our bound with the bound \eqref{DISCO} of the DISCO algorithm \cite{Zhang2018}, which unlike other works \cite{pmlr-v32-shamir14,reddi2016aide,JMLR:v21:19-764,pmlr-v119-hendrikx20a,dragomir2019optimal,sun2019convergence,hendrikx2020statistically,dvurechensky2021hyperfast} also achieves  the lower bound in terms of the dependence on $\beta$ and $\mu$. Since the dependence on these parameters in \eqref{DISCO} and in our bound \eqref{eq:ACRNM_complexity} are the same, we compare the other parts of the complexity bound. 

Let us denote $\Delta F(x_0) = F(x_0)-F(x_F^*)$. Using the fact, that a $\mu$-strongly-convex function with $L$-Lipschitz Hessian is self-concordant with constant $M = \dfrac{L}{(2\mu^{3/2})}$, we can rewrite our bound as
\begin{gather*}
O\left(\left(\frac{LR_0}{\mu}\right)^{1/3}\right) \leq  O \left(\left(\frac L {\mu}  \left( \frac { \Delta F(x_0)}\mu \right)^{1/2}\right)^{1/3}\right) \\
=    O\left( \left(\frac L {\mu^{3/2}}    \left( \Delta F(x_0)\right)^{1/2}\right) ^{1/3}\right) = O\left(( M^2 \Delta F(x_0))^{1/6}\right).
\end{gather*}
The corresponding part of the bound \eqref{DISCO} for the DISCO algorithm is much worse:
\[
  O\left(M^2 \Delta F(x_0)\sqrt{\frac{\beta}{\mu}}\right).
\]

%\todo{I think the comparison with DISCO should be here.}

\section{Application to Stochastic Optimization Problem}\label{sec:4}

In this section we return back to the stochastic optimization problem \eqref{eq:stoch_problem}. As it was described in subsection \ref{S:finite_sum_approx}, if the regularization parameter $\mu$ in \eqref{MC1} is chosen as $\mu = \widetilde{O} \left(\sqrt {\frac{L_0^2}{mn R^2}}\right)$ (Note that in this case $N=mn$, then an approximate solution to problem \eqref{MC1} is also an approximate solution to the stochastic optimization problem \eqref{eq:stoch_problem}. We  apply the algorithm from the previous section to solve problem \eqref{MC1} which satisfies assumptions of Theorem \ref{Th:ACRNM_conv_str_convex}. Combining the bound of this theorem with the bound \eqref{eq:bound_fin_sum_to_bound_stoch}, we obtain that the point $x^T$ generated by Algorithm \ref{alg:restarts} after $T$ rounds of communications satisfies
% \todo{Maybe get rig of $\delta$ by taking expectation or using the result of \url{https://arxiv.org/abs/1902.10710}.}
\begin{align*}\label{eq:rate_stoch}
 &\boldsymbol{F}(x^T) - \boldsymbol{F}(x^*) \leq \widetilde{O}\left(
 %\sqrt{\frac{L_0^2 R^2}{\delta N}} 
  L_0 R_02^{- \frac{T}{2\max\{\tau_1,\tau_2\}}} + {\frac{L_0 R}{\sqrt{N}}}\right)\\
 & = \widetilde{O}\left(L_0 R_02^{- \min\{T\left(\frac{\mu}{LR_{0}}\right)^\frac{1}{3},T\left(\frac{\mu}{\beta} \right)^\frac{1}{2}\} } + \sqrt{\frac{L_0^2 R^2}{N}}\right). 
\end{align*}

Further, substituting the value of $\mu$,  the value of $\beta$ from \eqref{eq:stat_sim} (see also Assumption \ref{as:similarity}), and omitting all the constants except $N$, $m$, $T$, $n=N/m$, we obtain %\todo{check}
\begin{equation}
    \begin{aligned}\label{eq:rate_stoch_simple}
     \boldsymbol{F}(x^T) - \boldsymbol{F}(x^*)  \leq& \\
     %\widetilde{O}\left( 2^{- \min\{T/N^\frac{1}{6},T/m^\frac{1}{4}\} } + \frac{1}{\sqrt{N}}\right)\\
     %&=O\left(\frac{1}{\sqrt{N}} + \exp\left(-C \min\{T/N^{1/6},T/m^{1/4}\}\right)\right)\\
      \exp\left(-\min\left\{\frac{C_1T}{N^{1/6}}, \frac{C_2T}{m^{1/4}}\right\}\right)& + \frac{C_3}{\sqrt{N}} %\widetilde{O}\left(\frac{1}{\sqrt{N}}\right), 
    \end{aligned}
\end{equation}
where constants $C_1, C_2, C_3$ depend on the parameters $L_0,L,R_0$ and logarithms of other parameters.

%\todo{Comparison with other offline bounds and online lower bounds.}

From the bound $\eqref{eq:rate_stoch_simple}$ we can obtain the dependence of the number of communication rounds $T$ and number of workers $m$ on the total number of observations $N$. We consider the case of full parallelization, i.e. when we use as many workers as possible and perform as less as possible communication rounds. The RHS of \eqref{eq:rate_stoch_simple} consists of two terms, and only the first one, that comes from the solution of the finite-sum approximation \eqref{MC1}, depends on $T$. So we would like to choose the number of communications such that both terms have the same order. Otherwise, the first term will either be larger than the second one, which means, that we have performed not enough communication rounds, or less, which implies that we have made too many communication rounds, and that does not improve the convergence. Therefore,  we get $T \simeq \max \{N^{1/6}, m^{1/4}\}$. Recall, that we also  would like to maximize $m$. If we choose $m \gtrsim N^{2/3}$, we will have $T \simeq m^{1/4}$. Hence, the number of communication rounds will increase with the number of workers. Therefore, the best possible choice is $m \simeq N^{2/3}$ and $T \simeq m^{1/4}$.

Communication requirements in terms of $T$ and $m$ for different approaches to solve $\eqref{StocProb}$ are presented in Table \ref{tab:comp}. One can see that our result is better than the lower bound for stochastic optimization $\eqref{Wood2021}$ in both $T$ and $m$. Compared to other state of the art approaches our method outperforms them either in number of communications or number of workers. 

In the case when the original stochastic problem $\eqref{eq:stoch_problem}$ is $\mu$-strongly convex, we no longer need to add regularization to have  convergence. From $\eqref{eq:bound_strong}$ and $\eqref{eq:ACRNM_coverg_T}$ we have
\begin{gather*}
     \EE \boldsymbol{F}(x^T) - \boldsymbol{F}(x^*) \leq 
      \\
      \exp\left(-\min \left\{C_1 T\mu^{1/3}, C_2T\mu^{1/2}n^{1/4}\right\}\right) + \left(\frac{C_3}{\mu N}\right),
\end{gather*}
where constants $C_1, C_2, C_3$ depend on the parameters $L_0,L,R_0$ and logarithms of other parameters.

\bibliographystyle{IEEEtran.bst}
\bibliography{literature,all_refs3}

\begin{thebibliography}{10}
\providecommand{\url}[1]{#1}
\csname url@rmstyle\endcsname
\providecommand{\newblock}{\relax}
\providecommand{\bibinfo}[2]{#2}
\providecommand\BIBentrySTDinterwordspacing{\spaceskip=0pt\relax}
\providecommand\BIBentryALTinterwordstretchfactor{4}
\providecommand\BIBentryALTinterwordspacing{\spaceskip=\fontdimen2\font plus
\BIBentryALTinterwordstretchfactor\fontdimen3\font minus
  \fontdimen4\font\relax}
\providecommand\BIBforeignlanguage[2]{{%
\expandafter\ifx\csname l@#1\endcsname\relax
\typeout{** WARNING: IEEEtran.bst: No hyphenation pattern has been}%
\typeout{** loaded for the language `#1'. Using the pattern for}%
\typeout{** the default language instead.}%
\else
\language=\csname l@#1\endcsname
\fi
#2}}

\bibitem{bor82}
V.~Borkar and P.~P. Varaiya, ``Asymptotic agreement in distributed
  estimation,'' \emph{IEEE Transactions on Automatic Control}, vol.~27, no.~3,
  pp. 650--655, 1982.

\bibitem{tsi84}
J.~N. Tsitsiklis and M.~Athans, ``Convergence and asymptotic agreement in
  distributed decision problems,'' \emph{IEEE Transactions on Automatic
  Control}, vol.~29, no.~1, pp. 42--50, 1984.

\bibitem{deg74}
M.~H. DeGroot, ``Reaching a consensus,'' \emph{Journal of the American
  Statistical Association}, vol.~69, no. 345, pp. 118--121, 1974.

\bibitem{xia06}
L.~Xiao and S.~Boyd, ``{Optimal scaling of a gradient method for distributed
  resource allocation},'' \emph{Journal of Optimization Theory and
  Applications}, vol. 129, no.~3, pp. 469--488, 2006.

\bibitem{rab04}
M.~Rabbat and R.~Nowak, ``Decentralized source localization and tracking
  wireless sensor networks,'' in \emph{Proceedings of the IEEE International
  Conference on Acoustics, Speech, and Signal Processing}, vol.~3, 2004, pp.
  921--924.

\bibitem{ram2009distributed}
S.~S. Ram, V.~V. Veeravalli, and A.~Nedic, ``Distributed non-autonomous power
  control through distributed convex optimization,'' in \emph{IEEE INFOCOM
  2009}.\hskip 1em plus 0.5em minus 0.4em\relax IEEE, 2009, pp. 3001--3005.

\bibitem{kra13}
T.~Kraska, A.~Talwalkar, J.~C. Duchi, R.~Griffith, M.~J. Franklin, and M.~I.
  Jordan, ``Mlbase: A distributed machine-learning system.'' in \emph{CIDR},
  vol.~1, 2013, pp. 2--1.

\bibitem{nedic2017fast}
A.~Nedi{\'c}, A.~Olshevsky, and C.~A. Uribe, ``Fast convergence rates for
  distributed non-bayesian learning,'' \emph{IEEE Transactions on Automatic
  Control}, vol.~62, no.~11, pp. 5538--5553, 2017.

\bibitem{woodworth2018graph}
B.~E. Woodworth, J.~Wang, A.~Smith, B.~McMahan, and N.~Srebro, ``Graph oracle
  models, lower bounds, and gaps for parallel stochastic optimization,'' in
  \emph{Advances in Neural Information Processing Systems}, 2018, pp.
  8505--8515.

\bibitem{kairouz2019advances}
P.~Kairouz, H.~B. McMahan, B.~Avent, A.~Bellet, M.~Bennis, A.~N. Bhagoji,
  K.~Bonawitz, Z.~Charles, G.~Cormode, R.~Cummings, \emph{et~al.}, ``Advances
  and open problems in federated learning,'' \emph{arXiv preprint
  arXiv:1912.04977}, 2019.

\bibitem{nemirovski2009robust}
\BIBentryALTinterwordspacing
A.~Nemirovski, A.~Juditsky, G.~Lan, and A.~Shapiro, ``Robust stochastic
  approximation approach to stochastic programming,'' \emph{SIAM Journal on
  Optimization}, vol.~19, no.~4, pp. 1574--1609, 2009. [Online]. Available:
  \url{https://doi.org/10.1137/070704277}
\BIBentrySTDinterwordspacing

\bibitem{shapiro2014lectures}
\BIBentryALTinterwordspacing
A.~Shapiro, D.~Dentcheva, and A.~Ruszczynski, \emph{Lectures on Stochastic
  Programming}.\hskip 1em plus 0.5em minus 0.4em\relax Society for Industrial
  and Applied Mathematics, 2009. [Online]. Available:
  \url{http://epubs.siam.org/doi/abs/10.1137/1.9780898718751}
\BIBentrySTDinterwordspacing

\bibitem{dvinskikh2020stochastic}
D.~Dvinskikh, ``Stochastic approximation versus sample average approximation
  for population wasserstein barycenters,'' \emph{Optimization methods and
  Software}, 2020.

\bibitem{li2020optimal}
H.~Li, Z.~Lin, and Y.~Fang, ``Optimal accelerated variance reduced extra and
  diging for strongly convex and smooth decentralized optimization,''
  \emph{arXiv preprint arXiv:2009.04373}, 2020.

\bibitem{woodworth2021minmax}
B.~Woodworth, B.~Bullins, O.~Shamir, and N.~Srebro, ``The min-max complexity of
  distributed stochastic convex optimization with intermittent communication,''
  2021.

\bibitem{woodworth2020minibatch}
B.~Woodworth, K.~K. Patel, and N.~Srebro, ``Minibatch vs local sgd for
  heterogeneous distributed learning,'' \emph{arXiv preprint arXiv:2006.04735},
  2020.

\bibitem{godichon2020rates}
A.~Godichon-Baggioni and S.~Saadane, ``On the rates of convergence of
  parallelized averaged stochastic gradient algorithms,'' \emph{Statistics},
  vol.~54, no.~3, pp. 618--635, 2020.

\bibitem{shalev2009stochastic}
S.~Shalev-Shwartz, O.~Shamir, N.~Srebro, and K.~Sridharan, ``Stochastic convex
  optimization.'' in \emph{COLT}, 2009.

\bibitem{daneshmand2021newton}
A.~Daneshmand, G.~Scutari, P.~Dvurechensky, and A.~Gasnikov, ``Newton method
  over networks is fast up to the statistical precision,'' 2021.

\bibitem{lan2020first}
G.~Lan, \emph{First-order and Stochastic Optimization Methods for Machine
  Learning}.\hskip 1em plus 0.5em minus 0.4em\relax Springer, 2020.

\bibitem{gorbunov2020local}
E.~Gorbunov, F.~Hanzely, and P.~Richt{\'a}rik, ``Local sgd: Unified theory and
  new efficient methods,'' \emph{arXiv preprint arXiv:2011.02828}, 2020.

\bibitem{nesterov2018lectures}
Y.~Nesterov, \emph{Lectures on convex optimization}.\hskip 1em plus 0.5em minus
  0.4em\relax Springer, 2018, vol. 137.

\bibitem{pmlr-v32-shamir14}
\BIBentryALTinterwordspacing
O.~Shamir, N.~Srebro, and T.~Zhang, ``Communication-efficient distributed
  optimization using an approximate newton-type method,'' in \emph{Proceedings
  of the 31st International Conference on Machine Learning}, ser. Proceedings
  of Machine Learning Research, E.~P. Xing and T.~Jebara, Eds., vol.~32,
  no.~2.\hskip 1em plus 0.5em minus 0.4em\relax Bejing, China: PMLR, 22--24 Jun
  2014, pp. 1000--1008. [Online]. Available:
  \url{http://proceedings.mlr.press/v32/shamir14.html}
\BIBentrySTDinterwordspacing

\bibitem{reddi2016aide}
S.~J. Reddi, J.~Kone{\v{c}}n{\`y}, P.~Richt{\'a}rik, B.~P{\'o}cz{\'o}s, and
  A.~Smola, ``Aide: Fast and communication efficient distributed
  optimization,'' \emph{arXiv preprint arXiv:1608.06879}, 2016.

\bibitem{JMLR:v21:19-764}
\BIBentryALTinterwordspacing
X.-T. Yuan and P.~Li, ``On convergence of distributed approximate newton
  methods: Globalization, sharper bounds and beyond,'' \emph{Journal of Machine
  Learning Research}, vol.~21, no. 206, pp. 1--51, 2020. [Online]. Available:
  \url{http://jmlr.org/papers/v21/19-764.html}
\BIBentrySTDinterwordspacing

\bibitem{pmlr-v119-hendrikx20a}
\BIBentryALTinterwordspacing
H.~Hendrikx, L.~Xiao, S.~Bubeck, F.~Bach, and L.~Massoulie, ``Statistically
  preconditioned accelerated gradient method for distributed optimization,'' in
  \emph{Proceedings of the 37th International Conference on Machine Learning},
  ser. Proceedings of Machine Learning Research, H.~D. III and A.~Singh, Eds.,
  vol. 119.\hskip 1em plus 0.5em minus 0.4em\relax PMLR, 13--18 Jul 2020, pp.
  4203--4227. [Online]. Available:
  \url{http://proceedings.mlr.press/v119/hendrikx20a.html}
\BIBentrySTDinterwordspacing

\bibitem{dragomir2019optimal}
R.-A. Dragomir, A.~Taylor, A.~d'Aspremont, and J.~Bolte, ``Optimal complexity
  and certification of bregman first-order methods,'' \emph{arXiv preprint
  arXiv:1911.08510}, 2019.

\bibitem{sun2019convergence}
Y.~Sun, A.~Daneshmand, and G.~Scutari, ``Convergence rate of distributed
  optimization algorithms based on gradient tracking,'' \emph{arXiv preprint
  arXiv:1905.02637}, 2019.

\bibitem{hendrikx2020statistically}
H.~Hendrikx, L.~Xiao, S.~Bubeck, F.~Bach, and L.~Massoulie, ``Statistically
  preconditioned accelerated gradient method for distributed optimization,''
  \emph{arXiv preprint arXiv:2002.10726}, 2020.

\bibitem{dvurechensky2021hyperfast}
P.~Dvurechensky, D.~Kamzolov, A.~Lukashevich, S.~Lee, E.~Ordentlich, C.~A.
  Uribe, and A.~Gasnikov, ``Hyperfast second-order local solvers for efficient
  statistically preconditioned distributed optimization,'' 2021.

\bibitem{arjevani2015communication}
Y.~Arjevani and O.~Shamir, ``Communication complexity of distributed convex
  learning and optimization,'' in \emph{Advances in neural information
  processing systems}, 2015, pp. 1756--1764.

\bibitem{Zhang2018}
\BIBentryALTinterwordspacing
Y.~Zhang and L.~Xiao, \emph{Communication-Efficient Distributed Optimization of
  Self-concordant Empirical Loss}.\hskip 1em plus 0.5em minus 0.4em\relax Cham:
  Springer International Publishing, 2018, pp. 289--341. [Online]. Available:
  \url{https://doi.org/10.1007/978-3-319-97478-1_11}
\BIBentrySTDinterwordspacing

\bibitem{nesterov2006cubic}
\BIBentryALTinterwordspacing
Y.~Nesterov and B.~Polyak, ``Cubic regularization of newton method and its
  global performance,'' \emph{Mathematical Programming}, vol. 108, no.~1, pp.
  177--205, 2006. [Online]. Available:
  \url{http://dx.doi.org/10.1007/s10107-006-0706-8}
\BIBentrySTDinterwordspacing

\bibitem{nesterov2008accelerating}
\BIBentryALTinterwordspacing
Y.~Nesterov, ``Accelerating the cubic regularization of newton's method on
  convex problems,'' \emph{Mathematical Programming}, vol. 112, no.~1, pp.
  159--181, Mar 2008. [Online]. Available:
  \url{https://doi.org/10.1007/s10107-006-0089-x}
\BIBentrySTDinterwordspacing

\bibitem{baes2009estimate}
\BIBentryALTinterwordspacing
M.~Baes, ``Estimate sequence methods:extensions and approximations,'' Tech.
  Rep., 2009. [Online]. Available:
  \url{http://www.optimization-online.org/DB_FILE/2009/08/2372.pdf}
\BIBentrySTDinterwordspacing

\bibitem{ghadimi2017second-order}
S.~Ghadimi, H.~Liu, and T.~Zhang, ``Second-order methods with cubic
  regularization under inexact information,'' \emph{arXiv:1710.05782}, 2017.

\bibitem{feldman2019high}
V.~Feldman and J.~Vondrak, ``High probability generalization bounds for
  uniformly stable algorithms with nearly optimal rate,'' in \emph{Conference
  on Learning Theory}.\hskip 1em plus 0.5em minus 0.4em\relax PMLR, 2019, pp.
  1270--1279.

\bibitem{zhang2018distributed}
G.~Zhang and R.~Heusdens, ``Distributed optimization using the primal-dual
  method of multipliers,'' \emph{IEEE Transactions on Signal and Information
  Processing over Networks}, vol.~4, no.~1, pp. 173--187, 2018.

\bibitem{hardt2016train}
M.~Hardt, B.~Recht, and Y.~Singer, ``Train faster, generalize better: Stability
  of stochastic gradient descent,'' in \emph{International Conference on
  Machine Learning}.\hskip 1em plus 0.5em minus 0.4em\relax PMLR, 2016, pp.
  1225--1234.

\bibitem{lei2020fine}
Y.~Lei and Y.~Ying, ``Fine-grained analysis of stability and generalization for
  stochastic gradient descent,'' in \emph{International Conference on Machine
  Learning}.\hskip 1em plus 0.5em minus 0.4em\relax PMLR, 2020, pp. 5809--5819.

\end{thebibliography}

\appendix

\subsection{SOTA approaches for distributed stochastic optimization}
\label{S:SOTA}

In the recent paper \cite{woodworth2021minmax} a novel lower bound for the SA approach was obtained:
\begin{equation}\label{eq:app_lb}
     \EE{\boldsymbol{F}(x^T)} - \boldsymbol{F}(x^*) \geq  \frac{LR^2}{\left(N/m\right)^2} + \frac{\sigma R}{\sqrt{N}} + \min\left\{\frac{LR^2}{T^2},\frac{\sigma R}{\sqrt{N/m}}\right\},
\end{equation}
where $N = nmT$, since in the SA approach we see each observation  once.
That bound is matched by a combination of two versions of Accelerated Gradient Descent \cite{woodworth2021minmax}. Recall that the convergence rate of batched accelerated SGD is 
$$
    O\left(\frac{LR^2}{t^2} + \frac{\sigma R}{\sqrt{t}r}\right),
$$
where $t$ is the number of iterations and $r$ is batch size. To obtain \eqref{Wood2021} one can consider two cases of the distributed setting:
\begin{itemize}
    \item single-machine $m=1,~t=\frac{N}{m},~r=1$:
    $$
        O\left(\frac{LR^2}{(N/m)^2} + \frac{\sigma R}{\sqrt{N/m}}\right);
    $$
    \item full batch $r = \frac{N}{T}, ~ t=T$:
    $$
        O\left(\frac{LR^2}{T^2} + \frac{\sigma R}{\sqrt{N}}\right).
    $$
\end{itemize}
The lower bound \eqref{eq:app_lb} is matched up to a logarithmic factor with the combination of these two regimes. 

There is also a lower bound for functions with Lipschitz Hessian , obtained by $\cite{woodworth2021minmax}$
$$
\frac{1}{\left(N/m\right)^2} + \frac{1}{\sqrt{N}} + \min\left\{\frac{1}{T^2},\frac{1}{\sqrt{N/m}},\frac{1}{\left(N/m\right)^{1/4}T^{7/4}}\right\}
$$
But it is not known whether it is accurate or not, since there is no method on which it reached.

In the paper \cite{godichon2020rates} authors get a better convergence rate considering non-accelerated parallelized SGD with specific step size and only one communication at the end,  assuming stronger local smoothness of objective near the solution
 $$O\left(\frac{1}{(N/m)} + \frac{1}{\sqrt{N}}\right).$$ 

Considering the SAA approach, one should note that optimization methods for original stochastic problem can also be applied. The most common example is stochastic gradient descent. Papers \cite{hardt2016train, lei2020fine} show that SGD, used to minimize the empirical risk of the model, also reduces the generalization error, if the number of iterations is not very large (linear in sample size $N$ \cite{hardt2016train}).

% This algorithm is $\e$-uniformly stable, which means that under the additional assumption of $L$-Lipschitz gradient the following inequality for the strongly convex $f(x, \xi)$ holds:
% $$
% |\frac{1}{N}\sum\limits_{i=1}^N f(x^K, \xi^i) - \boldsymbol{F}(x^K)| \leq \e = O\left(\frac{L_1^2}{\mu N}\right),
% $$
% where $x^K$ is the solution of SA problem \cite{hardt2016train, lei2020fine}. But still, in that approach and  the approach with regularization the number of data points should be huge. Moreover, we use them for the computation of stochastic gradients in different points $x$ a sufficiently large number of times. 

For the SAA approach Variance Reduction schemes can also be used.  In the paper \cite{woodworth2018graph} authors propose VR scheme that converges as follows
\begin{equation}\label{eq:app_vr_wood}
%2^{-\min\{mnT/N,nT/\sqrt{N}\}} + 1/ =
%   O(N^\frac{1}{4})
   \exp{\left(-\min \left \{\frac{C_1mnT}{N}, \frac{C_2L_0R}{L_1}\frac{nT}{\sqrt{N}} \right\}\right) } + \frac{C_3 L_0 R}{\sqrt{N}}, 
\end{equation}
where $C_1, C_2, C_3$ depends on logarithms of parameters $N, m, T$.
% Note, that in the formula above multiplier $O(N^{1/4})$ can be be moved to the exponent.

Accelerated VR algorithm from $\cite{li2020optimal}$ can be applied to the problem of the form \eqref{eq:MC1}.
% following form
% \begin{equation}\label{eq:app_saa2}
%     \frac{1}{m} \sum \limits_{i = 1}^m \frac{1}{K}\sum \limits_{j = 1}^K f (x, \xi^{i, j}) + \frac{\mu}{2}\|x\|^2 \to \min \limits_x.
% \end{equation}
In this case we have $N = nK$ total observations of stochastic gradient. For the convergence of that variance reduced method parameters $n, m, T$ must be selected in such a way that after $s$ iterations the following two conditions are satisfied
\begin{itemize}
    \item $ \sqrt{\dfrac{L}{\mu}} s \leq C_1 T ~~ \text{(stochastic updates)}$;
    \item $\left(K + \sqrt{K\dfrac{L}{\mu}}\right) s \leq C_2 nT  ~~ \text{(full gradient computation)}$
\end{itemize}
% \begin{itemize}
%     \item $T \geq \widetilde O\left(\sqrt{\frac{L}{\mu}}\right)$ rounds of communication,
%     \item $nT \geq \widetilde O\left(K + \sqrt{K\frac{L}{\mu}}\right)$ total computations of $\nabla f(x, \xi^{i, j})$ on each node
% \end{itemize}
to solve $\eqref{eq:MC1}$, where $C_1, C_2$ may depend on $K, n, m, T$ only logarithmically. Therefore, we get
$$
s \leq \min{\left\{\frac{C_1T}{\sqrt{L/\mu}}, \frac{C_2nT}{K + \sqrt{KL/\mu}}\right\}}.
$$
Using that we have fixed size number of observations $N = mK$ and $\mu = \widetilde{O}\left( \frac{L_0^2}{NR^2}\right)$, we obtain
$$s \leq \min{\left\{C_1\sqrt{\frac{L_0R}{L}}\frac{T}{N^{1/4}}, C_2\frac{mnT}{N}, C_2\sqrt{\frac{L_0R}{L}}\frac{\sqrt{m}nT}{N^{3/4}}\right\}}.$$
From \eqref{eq:bound_fin_sum_to_bound_stoch}, using the fact that convergence of this algorithm is linear, we have
% \aa{Finally, the convergence rate is
% $$\widetilde O(1) \left(\frac{1}{T^2} +\frac{1}{\sqrt{N}}  \right).$$}
\begin{equation}
    \begin{aligned}\label{eq:app_vr}
        \EE{\boldsymbol{F}(x^T)} - \boldsymbol{F}(x^*) \leq&  \left(\frac{1}{2}\right)^s + \widetilde{O}\left(\frac{L_0 R}{\sqrt N}\right) \leq \\
        \exp\left( - \min\left\{\sqrt{C_1\frac{L_0R}{L}}\frac{T}{N^{1/4}}\right.\right. &,\left.\left. C_2\frac{mnT}{N}, C_2\sqrt{\frac{L_0R}{L}}\frac{\sqrt{m}nT}{N^{3/4}}\right\} \right) \\
        + \widetilde{O}&\left(\frac{L_0 R}{\sqrt N}\right).
    \end{aligned}
\end{equation}

To compare these methods with the proposed one (Algorithms \ref{alg:ANPN}-\ref{alg:restarts}) and with the lower bound $\eqref{eq:app_lb}$ we derive dependence of parameters $T$ and $m$ on $N$, as we did before. Results are listed in the Table \ref{tab:comp}.

In the case of $\mu$-strong convexity of $f(x, \xi)$ w.r.t. $x$ for all $\xi$ offline bounds change since we don not need to regularize finite-sum approximation (see Section \ref{sec:2} for details). 
Therefore, convergence rate of Variance Reduction scheme  $\eqref{eq:app_vr_wood}$ from \cite{woodworth2018graph} changes to
\begin{gather*}
    \exp \left( - \min\left\{
    C_1\frac{nT}{L/\mu}, C_2\frac{mnT}{N}\right\} \right)+   O\left(\frac{L_0^2}{\mu N}\right)
\end{gather*}
And convergence rate $\eqref{eq:app_vr}$ of accelerated VR method \cite{li2020optimal} in the strongly convex case is
\begin{gather*}
    \exp \left( - \min \left \{
    \frac{C_1T}{\sqrt{L/\mu}},\frac{C_2T}{N/m}, \frac{C_2T}{\sqrt{(NL)/(m\mu)}}\right\} \right) +  O\left(\frac{L_0^2}{\mu N}\right).
\end{gather*}

\end{document}